\documentclass[10pt]{elsarticle}
\usepackage[cp1251]{inputenc}
\usepackage[english]{babel}
\usepackage{comment}
\usepackage{amsmath}
\usepackage{amssymb}
\usepackage{amsfonts}
\usepackage{amsthm}
\usepackage{mathtools}
\usepackage{dsfont}
\usepackage{rotating}
\usepackage{graphicx}
\usepackage{floatflt,epsfig}
\usepackage{lineno,hyperref}
\usepackage{enumerate}
\usepackage{colortbl}
\usepackage{enumitem}
\usepackage{array,tabularx,tabulary,booktabs}
\usepackage{longtable}
\usepackage{multirow}
\usepackage{wrapfig}
\usepackage{subcaption}
\usepackage{pdflscape}
\usepackage[table]{xcolor}
\usepackage[left=1in,right=1in,top=1in,bottom=1in]{geometry}
\newcolumntype{^}{>{\currentrowstyle}}

\journal{Arxiv}
\newtheorem{lemma}{Lemma}

\newtheorem{theorem}{Theorem}
\newtheorem{corollary}{Corollary}
\newtheorem{proposition}{Proposition}

\newtheorem{problem}{Problem}
\newtheorem{example}{Example}

\newcommand{\F}{\mathbb{F}}

\begin{document}
	\renewcommand{\abstractname}{Abstract}
	\renewcommand{\refname}{References}
	\renewcommand{\tablename}{Table}
	\renewcommand{\arraystretch}{0.9}
	\newcommand\qbin[3]{\left[\begin{matrix} #1 \\ #2 \end{matrix} \right]_{#3}}
	
	\thispagestyle{empty}
	\sloppy
	
	\begin{frontmatter}
		\title{Divisible design graphs from symplectic graphs over rings with precisely three ideals}
		
		\author[01]{Anwita Bhowmik}
		\ead{bhowmikanwita@gmail.com}
		
		\author[02]{Sergey Goryainov}
		\ead{sergey.goryainov3@gmail.com}

		\address[01] {Postdoctoral Research Station of Mathematics, School of Mathematical Sciences, Hebei Normal University, Shijiazhuang 050024, P.R. China}
		
		\address[02] {School of Mathematical Sciences, Hebei International Joint Research Center for Mathematics and Interdisciplinary Science, Hebei Key Laboratory of Computational Mathematics and Applications, Hebei Workstation for Foreign Academicians,\\Hebei Normal University, Shijiazhuang  050024, P.R. China}

		\begin{abstract}
		In this paper we construct two new infinite families of divisible design graphs based on symplectic graphs over rings with precisely three ideals.
		\end{abstract}
		
		\begin{keyword}
			divisible design graph; symplectic graph; group divisible design; ring; ideal 
			\vspace{\baselineskip}
			\MSC[2020] 		05B05  \sep 51E30 \sep 13H99 
		\end{keyword}
		
	\end{frontmatter}
	\section{Introduction}
	A \emph{divisible design graph} (a DDG for short) with parameters $(v,k,\lambda_1,\lambda_2,m,n)$ is a $k$-regular graph on $v$ vertices such that its vertex set can be partitioned into $m$ classes of size $n$ with the following two properties: any two distinct vertices from the same class have precisely $\lambda_1$ common neighbours and any two vertices from different classes have precisely $\lambda_2$ common neighbours. The partition from the definition of a DDG is called the \emph{canonical partition}. DDGs were introduced in \cite{M08} and \cite{HKM11} as a bridge between graph theory and design theory (the adjacency matrix of a divisible design graph can be regarded as the incidence matrix of a group divisible design \cite{B77}) and have been studied in \cite{CH14}, \cite{GHKS19}, \cite{KS21}, \cite{S21}, \cite{K22}, \cite{CS22}, \cite{PS22}, \cite{P22}, \cite{T22} \cite{K23}, \cite{GK24} and \cite{DGHS24}. DDGs can be viewed as an important subclass of Deza graphs \cite{EFHHH99}. The bibliography on strictly Deza graphs and DDGs, and databases of small examples can be found in \cite{P}.

    The family of symplectic strongly regular graphs \cite[Section 2.5]{BV22} is an important example of strongly regular polar graphs. Note that the vertex set of a symplectic graph is the set of 1-dimensional subspaces of a $2e$-dimensional vector space over a finite field for some integer $e \ge 2$. In the last two decades, there have been studies of the analogues of the strongly regular polar graphs (including the analogues of symplectic polar graphs) with the vector space to be replaced by a module over a finite ring (see \cite{MP11}, \cite{LWG13}, \cite{MP13}, \cite{G13}, \cite{GWZ13}, \cite{GLW14}, \cite{LGW14} for Deza graphs that can be constructed in this way). We also refer to \cite{LW08} and \cite{GLW10} for constructions of Deza graphs (the latter one generalises the former one) based on symplectic graphs over finite fields.

    In this paper, we work with the same setting of a symplectic space over a ring as given in \cite{MP13} and consider a certain class of rings. 
    	Throughout this paper, let $K$ be a finite commutative ring with identity, having precisely three ideals: the zero ideal, $K$ itself and $J=\langle r\rangle$. Without loss of generality we may assume that $|K|=q^2$, $|J|=q$ and $K/J \cong \mathbb{F}_q$, where $q$ is a prime power and $\mathbb{F}_q$ is the finite field of order $q$ (see Section \ref{sec:Prelim} for a background on rings). 
        Let $K^\times$ denote the set of units in $K$. Let $e\geq 2$ be an integer. Let $V'=\{(a_1,a_2,\ldots, a_{2e}): a_1,a_2,\ldots,a_{2e}\in K \text{ and } a_j\in K^\times\text{ for some }j\in \{1,2,\ldots, 2e\}\}$. We define an equivalence relation $\sim$ on $V'$ as: 
	$(a_1,a_2,\ldots,a_{2e})\sim (b_1,b_2,\ldots,b_{2e})$ if and only if there exists $\lambda\in K^\times$ such that $a_j=\lambda b_j, 1\leq j\leq 2e$. Let $V$ denote the set of equivalence classes on $V'$ corresponding to the equivalence relation $\sim$. We denote by $[a_1,a_2,\ldots, a_{2e}]$ the equivalence class of $(a_1,a_2,\ldots,a_{2e})$. Let $M$ be the $2e\times 2e$ non-singular matrix with entries in $K$, given by 
	\begin{align*}
		M=\begin{bmatrix}
			0 & I_e\\
			-I_e & 0\\
		\end{bmatrix}	
	\end{align*}
	where $I_e$ is the $e\times e$ identity matrix. 
    
    We define two graphs $X(2e,K)$ and $Y(2e,K)$, both with vertex set $V$:
    \begin{itemize}
    \item in the graph $X(2e,K)$, a vertex $[a_1,a_2,\ldots,a_{2e}]$ is adjacent to a vertex $[b_1,b_2,\ldots,b_{2e}]$ if and only if $(a_1,a_2,\ldots,a_{2e})M(b_1,b_2,\ldots,b_{2e})^t$ belongs to $K \setminus \{0\}$;
    \item 
    in the graph $Y(2e,K)$, a vertex $[a_1,a_2,\ldots,a_{2e}]$ is adjacent to a vertex $[b_1,b_2,\ldots,b_{2e}]$ if and only if $(a_1,a_2,\ldots,a_{2e})M(b_1,b_2,\ldots,b_{2e})^t$ belongs to $J \setminus \{0\}$. 
    \end{itemize}
    Clearly, $Y(2e,K)$ is a spanning subgraph of $X(2e,K)$.

\medskip
    The main results of this paper are the following two theorems.
	\begin{theorem}\label{thm:Main1}Let $e\geq 2$ be an integer.
		Let $K$ be a finite commutative ring with identity having precisely three ideals: $\langle 0\rangle$, $J$ and $K$. Let $K/J\cong \mathbb{F}_q$, where $\mathbb{F}_q$ is a finite field of order $q$ for a prime power $q$. Then the graph $X(2e,K)$ is a divisible design graph with parameters $(v,k,\lambda_1,\lambda_2,m,n)$ where
		\begin{align*}
			v&=\dfrac{q^{2e-1}(q^{2e}-1)}{q-1},
			~~k=q^{4e-2}+q^{4e-3}-q^{2e-2},
			~~\lambda_1=q^{4e-2}+q^{4e-3}-q^{4e-4}-q^{2e-2},\\
			\lambda_2&=q^{4e-2}+q^{4e-3}-q^{4e-4}-q^{4e-5}-q^{2e-2}+q^{2e-3},~~
			m=\dfrac{q^{2e}-1}{q-1},~~
			n=q^{2e-1}.
		\end{align*}
	\end{theorem}

	\begin{theorem}\label{thm:Main2}
    Let $e\geq 2$ be an integer.
		Let $K$ be a finite commutative ring with identity having precisely three ideals: $\langle 0\rangle$, $J$ and $K$. Let $K/J\cong \mathbb{F}_q$, where $\mathbb{F}_q$ is a finite field of order $q$ for a prime power $q$. Then the graph $Y(2e,K)$ is a divisible design graph with parameters $(v,k,\lambda_1,\lambda_2,m,n)$ where
		\begin{align*}
			v&=\dfrac{q^{2e-1}(q^{2e}-1)}{q-1},
			~~k=q^{4e-3}-q^{2e-2},
			~~\lambda_1=q^{4e-3}-q^{4e-4}-q^{2e-2},\\
			\lambda_2&=q^{4e-4}-q^{4e-5}-q^{2e-2}+q^{2e-3},
			~~m=\dfrac{q^{2e}-1}{q-1},
			~~n=q^{2e-1}.
		\end{align*}
	\end{theorem}

\begin{example}
Let $p$ be a positive integer that is a prime. Then the ring  $K_1 = \mathbb{Z}/p^2\mathbb{Z}$ has precisely three ideals, where $|K_1| = p^2$ and the non-trivial ideal $J$ is given by the elements divisible by $p$ (if we represent $K_1$ as the set of remainders modulo $p^2$).        
\end{example}

\begin{example}
Let $q$ be a positive integer that is a prime power and $\mathbb{F}_q[x]$ be the ring of polynomials of finite degree over the finite field $\mathbb{F}_q$. Then the quotient ring $K_2 = \mathbb{F}_q[x]/ \langle x^2\rangle$ has precisely three ideals, where $|K_2| = q^2$ and the non-trivial ideal $J$ is given by all the elements divisible by $x$ (if we represent $K_2$ as the set of remainders modulo $x^2$).     
\end{example}

Note that if $q = p$, then the rings $K_1$ and $K_2$ have the same order, but are not isomorphic (as they have different characteristic); in this case the graphs $X(2e,K_1)$ and $X(2e,K_2)$ have the same parameters and the graphs $Y(2e,K_1)$ and $Y(2e,K_2)$ also have the same parameters. We have verified computationally that for $q = p = 2$ the graphs $X(4,K_1)$ and  $X(4,K_2)$ are not isomorphic and the graphs $Y(4,K_1)$ and  $Y(4,K_2)$ are also not isomorphic.

We do not know any example of a finite commutative ring with identity, having precisely three ideals and being not isomorphic to $K_1$ or $K_2$. In particular, the classification of small local rings \cite{N18} does not give such examples.
    
The paper is organised as follows. In Section \ref{sec:Prelim} we give some preliminary definitions and results. In section \ref{sec:Proofs}, we give proofs of Theorem \ref{thm:Main1} and Theorem \ref{thm:Main2} in a series of lemmas and propositions. In Section \ref{sec:Remarks}, we give some concluding remarks.
	\section{Preliminaries}\label{sec:Prelim}
In this section, we give preliminary definitions and results from ring theory.
    
	A \emph{ring} $R$ is a set together with two binary operations $+$ and $\cdot$ (called \emph{addition} and \emph{multiplication}) satisfying the following axioms:
	\begin{itemize}
		\item $(R,+)$ is an abelian group (with additive identity denoted by $0$),
		\item the operation $\cdot$ is associative: $(a\cdot b)\cdot c=a\cdot (b\cdot c)$ for all $a,b,c\in R$, and
		\item the distributive laws hold in $R$: for all $a,b,c\in R$,
		\begin{align*}
			(a+b)\cdot c=a\cdot c+b\cdot c\text{ and }a\cdot (b+c)=a\cdot b+a\cdot c.
		\end{align*}
	\end{itemize}
	We write $ab$ instead of $a\cdot b$ as shorthand.
	
	A ring $R$ is called \emph{commutative} if multiplication is commutative, that is, $ab=ba$ for all $a,b\in R$.
	
	A ring $R$ is said to have an \emph{identity} if there is an element $1\in R$ with $1\cdot a=a\cdot 1=a$ for all $a\in R$.

    Let $R$ be a ring with identity $1\neq 0$. An element $a$ of $R$ is called a \emph{unit} (or $a$ is said to be \emph{invertible}) in $R$ if there is some $b$ in $R$ such that $ab=ba=1$. The set of units in $R$ is denoted by $R^\times$.
	
	A nonzero element $a\in R$ is called a \emph{zero divisor} if there is a nonzero element $b\in R$ such that either $ab=0$ or $ba=0$.

A \emph{subring} of the ring $R$ is a subset of $R$ that is itself a ring when binary operations of addition and multiplication on $R$ are restricted to the subset.

	Let $R$ and $S$ be rings. A \emph{ring homomorphism} is a map $\phi:R\rightarrow S$ satisfying
	\begin{itemize}
		\item $\phi(a+b)=\phi(a)+\phi(b)$ and
		\item $\phi(ab)=\phi(a)\phi(b)$
	\end{itemize} 
for all $a,b\in R$. A bijective ring homomorphism is called an \emph{isomorphism}.

\begin{proposition}[{\cite[Exercise 17, p. 231]{DF04}}]
	Let $R$ and $S$ be rings. Then the direct product $R\times S$ is a ring under componentwise addition and multiplication.
\end{proposition}

Now, let $R$ be a commutative ring with identity.

The ring $R$ is called an \emph{integral domain} if $1\neq 0$ and $R$ has no zero divisors.
	
A subset $I$ of $R$ is called an \emph{ideal} of $R$ if $I$ is a subring of $R$ and $I$ is closed under multiplication by elements of $R$, that is, $aI\subseteq I$ for all $a\in R$.

Let $A$ be a subset of $R$. Let $\langle A\rangle$ denote the smallest ideal of $R$ containing $A$; this ideal is called the \emph{ideal generated by $A$}.

An ideal $I$ of $R$ is called a \emph{principal ideal} if $I$ is generated by a single element, that is, if there exists $a\in R$ such that $I=\langle a\rangle=\{ab:b\in R\}$.

    Let $a,b\in R$. We say that \emph{$a$ divides $b$} if there exists $c\in R$ such that $b=ac$. Note that $a$ divides $b$ if and only if $\langle b\rangle\subseteq\langle a\rangle$.

Let $I$ be an ideal of $R$. Then the (additive) \emph{quotient ring of $R$ by $I$}, denoted by $R/I$, is the ring under the binary operations:
	$$(a+I)+(b+I):=(a+b)+I \text{ and } (a+I)\cdot (b+I) :=ab+I$$ for all $a,b\in R$.

\begin{theorem}[{\cite[Theorem 8, p. 246]{DF04}}]
	Let $R$ be a commutative ring with identity, and let $I$ be an ideal of $R$. The correspondence $A\xleftrightarrow{}A/I$ is an inclusion preserving bijection between the set of subrings $A$ of $R$ that contain $I$ and the set of subrings of $R/I$. Furthermore, $A$ (a subring containing $I$) is an ideal of $R$ if and only if $A/I$ is an ideal of $R/I$.
\end{theorem}

An ideal $\mathfrak{m}$ in $R$ is called a \emph{maximal ideal} if $\mathfrak{m}\neq R$ and the only ideals containing $\mathfrak{m}$ are $\mathfrak{m}$ and $R$.

\begin{proposition}[{\cite[Proposition 12, p. 254]{DF04}}]
	Let $R$ be a commutative ring with identity. The ideal $\mathfrak{m}$ is a maximal ideal if and only if the quotient ring $R/\mathfrak{m}$ is a field.
\end{proposition}

The ring $R$ is called a \emph{local ring} if it has a unique maximal ideal.

\begin{proposition}[{\cite[Exercise 37, p. 259]{DF04}}]
	Let $R$ be a commutative ring with identity. If $R$ is a local ring with maximal ideal $\mathfrak{m}$, then every element of $R\setminus\mathfrak{m}$ is a unit. Thus, $R\setminus\mathfrak{m}$ equals the set of units in $R$.
\end{proposition}

	The ring $R$ is called a \emph{principal ideal ring} if every ideal of $R$ is principal.
	
	The ring $R$ is called a \emph{principal ideal domain (PID)} if it is an integral domain which is a principal ideal ring.

     An ideal $P$ of $R$ is called a \emph{prime ideal} if $P\neq R$ and whenever the product $ab$ of two elements $a,b\in R$ is an element of $P$, then at least
one of $a$ and $b$ is an element of $P$. 

	Let $R$ be an integral domain. For an element $a\in R$ such that $a$ is nonzero and a non-unit,
	\begin{itemize}
		\item $a$ is called \emph{irreducible} in $R$ if $a=bc$ with $b,c\in R$ implies $a\in R^\times$ or $b\in R^\times$, and
		\item $a$ is called \emph{prime} in $R$ if the ideal $\langle a\rangle$ is a prime ideal.
        \end{itemize}
Two elements $a$ and $b$ of $R$ are said to be \emph{associate} in $R$ if $a=bc$ for some unit $c$ in $R$ (note that in this case, $\langle a\rangle=\langle b\rangle$).

A \emph{unique factorisation domain (UFD)} is an integral domain $R$ in which every nonzero element $a\in R$ which is not a unit has the following two properties:
\begin{itemize}
	\item $a$ can be written as a finite product of irreducibles in $R$: $a=p_1\cdots p_k$ where the factors are irreducibles in $R$, and
	\item the decomposition is unique up to associates: that is, if $a=p_1'\cdots p_{\ell}'$ is another factorisation of $a$ into irreducibles, then $k=\ell$ and there is some renumbering of the factors so that $p_j$ is associate to $p_j'$ for $j=1,\ldots,k$.
\end{itemize}

\begin{proposition}[{\cite[Proposition 12, p. 286]{DF04}}] 
	In a UFD a nonzero non-unit element is a prime if and only if it is irreducible. 
\end{proposition}

\begin{theorem}[{\cite[Theorem 14, p. 287]{DF04}}]
	Every PID is a UFD.
\end{theorem}

\begin{proposition}
	Let $R$ be a finite commutative ring with identity. Then every element of $R$ is either a unit or a zero divisor.
\end{proposition}

\begin{proof}
	Let $a$ be a nonzero non-unit element of $R$. Then $\langle a\rangle\neq R$, and the function
	\begin{align*}
		\phi&: R\rightarrow\langle a\rangle\\
		&x\mapsto ax
	\end{align*}
	cannot be injective. So, there exists $x\neq y$ such that $ax=ay$. Thus, we find that $a(x-y)=0$ where $x-y\neq 0$, wherby $a$ is a zero divisor.
\end{proof}

		Let $I$ be an ideal of $R$. Let $a,b\in R$. We write $a\equiv b\pmod J$ if $a-b\in J$.

	For the ring $R$, we call an ideal $J$ \emph{non-trivial} if it is a nonzero ideal which is not equal to $R$.

	\begin{lemma}
		Let $K$ be a finite commutative ring with identity having precisely three ideals. Then $K\cong P/\langle p^2\rangle$ for some principal ideal domain $P$ and prime $p\in P$.
	\end{lemma}
	\begin{proof}
		$K$ is a principal ideal ring, so by \cite[Theorem 1]{H68}, it is a finite direct product of quotients of PIDs. Let $K\cong P_1/Q_1\times \cdots\times P_i/Q_i$ where $i\geq 1$, and for $j\in\{1,\ldots,i\}$, $P_j$ is a PID and $Q_j$ is an ideal of $P_j$. But $K$ has precisely three ideals, so $i=1$ and therefore $K\cong P/Q$ for some PID $P$ and some ideal $Q$ of $P$. Let $Q=\langle x\rangle$ for some $x\in Q$. Then, ideals of $P/Q$ are of the form $Q_1/Q$ where $Q_1$ is an ideal containing $Q$. If $Q_1=\langle y\rangle$, then $Q_1$ contains $Q$ if and only if $y$ divides $x$. So, for $P/Q$ to have precisely three ideals, $x$ must have exactly three divisors in $P$, unique up to units. If $x=0$, then $K$ must have exactly three elements, unique up to units, which is not possible. Therefore $x\neq 0$; let $x=p_1^{h_1}\cdots p_i^{h_i}$ (unique up to units), $i\geq 1$. Then $x$ has $(h_1+1)\cdots (h_i+1)$ divisors, so $(h_1+1)\cdots (h_i+1)=3$ implies $i=1, h_1=2$. This completes the proof.
	\end{proof}
	\begin{lemma}
		Let $K$ be a finite commutative ring having identity and a unique non-trivial ideal $J$. Then $K/J\cong \mathbb{F}_q$, where $\mathbb{F}_q$ is a finite field of order $q$ for a prime power $q$.
	\end{lemma}
	\begin{proof}
		The proof follows from the fact that the ideals of $K$ are $\langle 0\rangle \subsetneq J\subsetneq K$, whereby $K$ is a local ring with unique maximal ideal $J$.
	\end{proof}
	\begin{lemma}\label{lemij=0}
		Let $K$ be a finite commutative ring having identity and a unique non-trivial ideal $J$. Then, for any $x,y\in J$, the equality $xy=0$ holds.
	\end{lemma}
	\begin{proof}
		The result holds if $x=0$ or $y=0$, so assume that $x\neq 0, y\neq 0$. Suppose $xy\neq 0$. Then $\langle xy\rangle=J$ or $\langle xy\rangle=K$. If $\langle xy\rangle=K$ then $1\in K=\langle xy\rangle$, so $xyw=1$ for some $w\in K$, which implies that $x$ is a unit, so $x\notin J$, a contradiction. Therefore, $\langle xy\rangle=J$. Then, $xyz=x$ for some $z\in K$, that is, $x(yz-1)=0$. But $y\in J$ implies $yz\in J$. Since $J$ is a non-trivial ideal of $K$, $1$ does not belong to $J$. Then $(yz-1)\notin J$, so $(yz-1)$ is a unit. which implies $x=0$, a contradiction. This completes the proof.
	\end{proof}
	\begin{lemma}\label{lemjr}
		Let $K$ be a finite commutative ring with identity having a unique non-trivial ideal $J=\langle r\rangle$, such that $K/J\cong \mathbb{F}_q$, where $\mathbb{F}_q$ is a finite field of order $q$ for a prime power $q$. Let $K/J=\{z_1+J,z_2+J,\ldots, z_q+J\}$ where $\{z_1,z_2,\ldots,z_q\}$ is a set of coset representatives. Then $J=\{z_1r,z_2r,\ldots, z_qr\}$ and $K=\{z_{j_1}+z_{j_2}r: j_1,j_2\in\{1,2,\ldots,q\}\}$, where $|J|=q$ and $|K|=q^2$.
	\end{lemma}
	\begin{proof}
		Evidently, $\{z_1r,z_2r,\ldots, z_qr\}\subseteq J$. Let $1\leq j_1<j_2\leq q$, and suppose $z_{j_1}r=z_{j_2}r$. This implies $(z_{j_1}-z_{j_2})r=0$. Since $z_{j_1}+J$ and $z_{j_2}+J$ are distinct elements of $K/J$, $(z_{j_1}-z_{j_2})$ does not belong to $J$, and so $(z_{j_1}-z_{j_2})$ is a unit. This implies $r=0$ which is not possible. Therefore, $J=\{z_1r,z_2r,\ldots, z_qr\}$. Next, we observe that $|K|=|K/J|\cdot|J|=q^2$, so it suffices to show that the elements of $\{z_{j_1}+z_{j_2}r: j_1,j_2\in\{1,2,\ldots,q\}\}$ are distinct. Let $z_{j_1}+z_{j_2}r=z_{j_3}+z_{j_4}r$ for some $j_1,j_2,j_3,j_4\in\{1,2,\ldots,q\}$. Then $z_{j_1}-z_{j_3}=(z_{j_2}-z_{j_4})r$, and since $(z_{j_2}-z_{j_4})r\in J$, we conclude $z_{j_1}=z_{j_3}$. So, $(z_{j_2}-z_{j_4})r=0$. If $z_{j_2}\neq z_{j_4}$ then $z_{j_4}+J\neq z_{j_2}+J$. So, $(z_{j_2}-z_{j_4})$ is a unit, and, consequently, $r=0$ which is not possible. This completes the proof.
	\end{proof}
	Note that if $K/J=\{z_1+J,z_2+J,\ldots, z_q+J\}$ where $\{z_1,z_2,\ldots,z_q\}$ is a set of coset representatives, then for $j_1,j_2\in\{1,2,\ldots,q\}$, $(z_{j_1}-z_{j_2})$ is invertible in $K$ if and only if $j_1\neq j_2$.

	\begin{lemma}
		Let $K$ be a finite commutative ring with identity having a unique non-trivial ideal $J=\langle r\rangle$, such that $K/J\cong \mathbb{F}_q$, where $\mathbb{F}_q$ is a finite field of order $q$ for a prime power $q$. Let $K/J=\{z_1+J,z_2+J,\ldots, z_q+J\}$ where $\{z_1,z_2,\ldots,z_q\}$ is a set of coset representatives with $z_{1}=0$. Let $j_1,j_2\in\{1,2,\ldots,q\}$. Then the following hold.
		\begin{enumerate}[label=\textnormal{(\arabic*)},ref=\thelemma (\arabic*)]
			\item \label{lempart1}$z_{j_1}+z_{j_2}r\in J$ if and only if $z_{j_1}=0$.
			\item \label{lempart2} $z_{j_1}+z_{j_2}r=0$ if and only if $z_{j_1}=z_{j_2}=0$.
		\end{enumerate}
	\end{lemma}
	\begin{proof}
		The proofs readily follow by noting that for $j\in \{1,2,\ldots,q\}$, $z_{j}$ belongs to $J$ if and only if $z_{j}=0$.
	\end{proof}

        \section{Proofs of Theorem \ref{thm:Main1} and Theorem \ref{thm:Main2}}\label{sec:Proofs}

In this section, we give proofs of Theorem \ref{thm:Main1} and Theorem \ref{thm:Main2} in a series of lemmas and propositions.

For a graph $G$, we denote the complementary graph of $G$ by $\overline{G}$.
        \subsection{The number of vertices and the degrees of the graphs $X(2e,K)$ and $Y(2e,K)$}
In this section, we determine the number of vertices and the degrees of the graphs $X(2e,K)$ and $Y(2e,K)$.

	\begin{proposition}\label{prop1}
    Let $e\geq 2$ be an integer.
		Let $K$ be a finite commutative ring with identity having a unique non-trivial ideal $J$ such that $K/J\cong \mathbb{F}_q$, where $q$ is a prime power.  The number of vertices in each of the graphs $X(2e,K)$ and $Y(2e,K)$ is $\dfrac{q^{2e-1}(q^{2e}-1)}{q-1}$.  
	\end{proposition}
	\begin{proof}
		The number of elements in $V'$ is $(q^2)^{2e}-q^{2e}$. We find that the number of elements in an equivalence class corresponding to the relation $\sim$ is the number of elements in $K^\times$, by noting that $\lambda_1 (x_1,x_2,\ldots, x_{2e})=\lambda_2 (x_1,x_2,\ldots, x_{2e})$ for some $\lambda_1,\lambda_2\in K^\times$ and $(x_1,x_2,\ldots,x_{2e})\in V'$ implies that $(\lambda_1-\lambda_2)x_j=0$ for $1\leq j\leq 2e$, so if $x_i\in K^\times$ for some $i\in\{1,2,\ldots,2e\}$ then $(\lambda_1-\lambda_2)x_i=0$ implies that $\lambda_1=\lambda_2$. So, the number of vertices in the graph equals the number of equivalence classes corresponding to $\sim$, that is, 
		$$\dfrac{(q^2)^{2e}-q^{2e}}{q^2-q}
		=
		\dfrac{q^{2e-1}(q^{2e}-1)}{q-1}. 
		$$
	\end{proof}
	
	\begin{proposition}\label{prop2}
    Let $e\geq 2$ be an integer.
		Let $K$ be a finite commutative ring with identity having a unique non-trivial ideal $J$ such that $K/J\cong \mathbb{F}_q$, where $q$ is a prime power.  Then the graph $\overline{X(2e,K)}$ is regular of degree $\dfrac{q^{4e-3}-q^{2e-2}-q+1}{q-1}$, and the graph $X(2e,K)$ is regular of degree $q^{4e-2}+q^{4e-3}-q^{2e-2}$.
	\end{proposition}
	\begin{proof}
		Let us fix a vertex $[a_1,a_2,\ldots,a_{2e}]$ of the graph $\overline{X(2e,K)}$. We find the number of possibilities for the vertex $[b_1,b_2,\ldots,b_{2e}]$ such that $[b_1,b_2,\ldots,b_{2e}]$ is adjacent to $[a_1,a_2,\ldots,a_{2e}]$ in $\overline{X(2e,K)}$. To begin with, we find the number of solutions $(b_1,b_2,\ldots, b_{2e})$ such that
		\begin{align}\label{eq0}
			&(a_1b_{e+1}-a_{e+1}b_1) + (a_2b_{e+2}-a_{e+2}b_2) + \cdots + (a_eb_{2e}-a_{2e}b_e )=0.  
		\end{align}
		Since $(a_1,a_2,\ldots,a_{2e})\in V'$, there exists $i\in\{1,2,\ldots,2e\}$ such that $a_i\in K^\times$. If $i\leq e$, then
		\begin{align}\label{eqless}
			b_{e+i} &= a_i^{-1}((a_{e+1}b_1-a_1b_{e+1}) + (a_{e+2}b_2-a_2b_{e+2}) \notag\\
			&+\cdots+ (a_{e+i-1}b_{i-1}-a_{i-1}b_{e+i-1})
			+a_{e+i}b_i + (a_{e+i+1}b_{i+1}- a_{i+1}b_{e+i+1})\notag \\
			&+ \cdots + (a_{2e}b_e- a_eb_{2e} )),
		\end{align}
		and if $i\geq e+1$, then
		\begin{align}\label{eqmore}
			b_{i-e} &= a_i^{-1}((a_{1}b_{e+1}-a_{e+1}b_{1}) + (a_{2}b_{e+2}-a_{e+2}b_{2}) \notag\\
			&+\cdots+ (a_{i-1-e}b_{i-1}-a_{i-1}b_{i-1-e})
			+a_{i-e}b_i + (a_{i+1-e}b_{i+1}- a_{i+1}b_{i+1-e})\notag \\
			&+ \cdots + (a_{e}b_{2e}- a_{2e}b_{e} )).
		\end{align}
		This shows that there are $2e-1$ choices of $b_j$'s possible. Since $(b_1,b_2,\ldots,b_{2e})\in V'$, it must have an invertible entry. When $i\leq e$ (resp. $i\geq e+1$), if all $b_j$'s to the right hand side of Equation \eqref{eqless} (resp. \eqref{eqmore}) are elements in $J$, then $b_{e+i}$ (resp. $b_{i-e}$) is also an element in $J$, so $(b_1,b_2,\ldots,b_{2e})$ does not belong to $ V'$. On the other hand, when $i\leq e$ (resp. $i\geq e+1$), if there exists $j\neq e+i$ (resp. $j\neq i-e$) such that $b_j\in K^\times$, then $(b_1,b_2,\ldots,b_{2e})$ belongs to $ V'$. So, to count the number of tuples $(b_1,b_2,\ldots,b_{2e})\in V'$ satisfying Equation \eqref{eq0}, it is enough to count the number of tuples $(b_1,b_2,\ldots,b_{2e})\in \underbrace{K\times K\times \cdots\times K}_{2e\text{ times}}$ satisfying Equation \eqref{eq0} such that for all $j\neq e+i$ when $i\leq e$ (resp. $j\neq i-e$ when $i\geq e+1$), $b_j$ does not belong to $J$. Moreover, Equations \eqref{eqless} and \eqref{eqmore} are satisfied by putting $b_j=a_j$ for all $j\in\{1,2,\ldots,2e\}$. So, the number of possibilities for the vertex $[b_1,b_2,\ldots,b_{2e}]$ being adjacent to $[a_1,a_2,\ldots,a_{2e}]$ such that Equation \eqref{eq0} holds is given by
		\begin{align*}
			&\dfrac{\text{the number of tuples }(b_1,b_2,\ldots,b_{2e})\in V'\text{ satisfying Equation }\eqref{eq0} }{\text{size of an equivalence class corresponding to }\sim}\\
			&-\text{ the possibility that }[b_1,b_2,\ldots,b_{2e}]=[a_1,a_2,\ldots,a_{2e}]\\
			&=\dfrac{(q^2)^{2e-1}-q^{2e-1}}{q^2-q}-1\\
			&= \dfrac{q^{4e-3}-q^{2e-2}-q+1}{q-1},
		\end{align*}
		which equals the degree of the graph $\overline{X(2e,K)}$. 
		Then, $X(2e,K)$ is regular of degree
		\begin{align*}
			&\dfrac{q^{2e-1}(q^{2e}-1)}{q-1}-1-\dfrac{q^{4e-3}-q^{2e-2}-q+1}{q-1}\\
			=&\dfrac{q^{4e-1}-q^{2e-1}-q^{4e-3}+q^{2e-2}+q-1}{q-1} - 1 \\
			=& q^{4e-2}+q^{4e-3}-q^{2e-2},
		\end{align*}
	and the proof is complete.	
	\end{proof}
	\begin{proposition}\label{prop3}
    Let $e\geq 2$ be an integer.
		Let $K$ be a finite commutative ring with identity having a unique non-trivial ideal $J=\langle r\rangle$ such that $K/J\cong \mathbb{F}_q$ for a prime power $q$. Let $e\geq 1$ be an integer. Then $Y(2e,K)$ is regular of degree $q^{4e-3}-q^{2e-2}$.
		
	\end{proposition}
	\begin{proof}
		Let $K/J=\{z_1+J,z_2+J,\ldots, z_q+J\}$ where $\{z_1,z_2,\ldots,z_q\}$ is a set of coset representatives with $z_1=0$. Then $J=\{z_1r,z_2r,\ldots, z_qr\}$.	Now, let us fix a vertex $[a_1,a_2,\ldots,a_{2e}]$ of the graph $Y(2e,K)$. We find the number of possibilities of the vertex $[b_1,b_2,\ldots,b_{2e}]$ such that $[b_1,b_2,\ldots,b_{2e}]$ is adjacent to $[a_1,a_2,\ldots,a_{2e}]$ in $Y(2e,K)$, that is,
		\begin{align}\label{eq01}
			&(a_1b_{e+1}-a_{e+1}b_1) + (a_2b_{e+2}-a_{e+2}b_2) + \cdots + (a_eb_{2e}-a_{2e}b_e )=z_jr 
		\end{align}
		for some $j \in \{2,\ldots, q\}$.
		Since $(a_1,a_2,\ldots,a_{2e})\in V'$, there exists $i\in\{1,2,\ldots,2e\}$ such that $a_i\in K^\times$. If $i\leq e$, then
		\begin{align}\label{eqless1}
			b_{e+i} &= a_i^{-1}(z_jr+(a_{e+1}b_1-a_1b_{e+1}) + (a_{e+2}b_2-a_2b_{e+2}) \notag\\
			&+\cdots+ (a_{e+i-1}b_{i-1}-a_{i-1}b_{e+i-1})
			+a_{e+i}b_i + (a_{e+i+1}b_{i+1}- a_{i+1}b_{e+i+1})\notag \\
			&+ \cdots + (a_{2e}b_e- a_eb_{2e} )),
		\end{align}
		and if $i\geq e+1$, then
		\begin{align}\label{eqmore1}
			b_{i-e} &= a_i^{-1}((a_{1}b_{e+1}-a_{e+1}b_{1}) + (a_{2}b_{e+2}-a_{e+2}b_{2}) \notag\\
			&+\cdots+ (a_{i-1-e}b_{i-1}-a_{i-1}b_{i-1-e})
			+a_{i-e}b_i + (a_{i+1-e}b_{i+1}- a_{i+1}b_{i+1-e})\notag \\
			&+ \cdots + (a_{e}b_{2e}- a_{2e}b_{e} )-z_jr).
		\end{align}
		This shows that there are $2e-1$ choices of $b_j$'s possible. Since $(b_1,b_2,\ldots,b_{2e})\in V'$, it must have an invertible entry. When $i\leq e$ (resp. $i\geq e+1$), if all $b_j$'s to the right hand side of Equation \eqref{eqless1} (resp. \eqref{eqmore1}) are elements in $J$, then $b_{e+i}$ (resp. $b_{i-e}$) is also an element in $J$, so $(b_1,b_2,\ldots,b_{2e})\notin V'$. On the other hand, when $i\leq e$ (resp. $i\geq e+1$), if there exists $j\neq e+i$ (resp. $j\neq i-e$) such that $b_j\in K^\times$, then $(b_1,b_2,\ldots,b_{2e})\in V'$. So, to count the number of tuples $(b_1,b_2,\ldots,b_{2e})\in V'$ satisfying Equation \eqref{eq01}, it is enough to count the number of tuples $(b_1,b_2,\ldots,b_{2e})\in \underbrace{K\times K\times \cdots\times K}_{2e\text{ times}}$ satisfying Equation \eqref{eq01} such that for all $j\neq e+i$ when $i\leq e$ (resp. $j\neq i-e$ when $i\geq e+1$), $b_j$ does not belong to $J$.
		Therefore, the number of possibilities for the vertex $[b_1,b_2,\ldots,b_{2e}]$ being adjacent to $[a_1,a_2,\ldots,a_{2e}]$ such that Equation \eqref{eq01} holds is given by
		\begin{align*}
			&\dfrac{(\text{the number of choices of }k)\times	(	\text{the number of tuples }(b_1,b_2,\ldots,b_{2e})\in V'\text{ satisfying Equation }\eqref{eq01})}{\text{size of an equivalence class corresponding to }\sim}\\
			&=\dfrac{(q-1)((q^2)^{2e-1}-q^{2e-1})}{q^2-q}\\
			&= q^{4e-3}-q^{2e-2}.
		\end{align*}  
		This completes the proof.
	\end{proof}
	
	\subsection{The canonical partition of the graphs $X(2e,K)$ and $Y(2e,K)$}
In this section, we describe the canonical partition of the graphs $X(2e,K)$ and $Y(2e,K)$.

	We recall that $J=\langle r\rangle$ is the unique non-trivial ideal of $K$ such that $|J|=q$ and $|K|=q^2$ for some prime power $q$, where $K/J=\{z_1+J,z_2+J,\ldots, z_q+J\}$ and $\{z_1,z_2,\ldots,z_q\}$ is a set of coset representatives with $z_1=0$. Put $T:=\underbrace{J \times J \times \ldots \times J}_{2e \text{ times}}$. Then $T$ is a vector space over $K/J$ of dimension $2e$, with the  scalar multiplication defined as,
	$$(z_{j}+J).(x_1,x_2,\ldots,x_{2e})=(z_jx_1,z_j x_2,\ldots,z_jx_{2e}),$$ where $j\in\{1,2,\ldots,q\}$ and $x_1,x_2,\ldots,x_{2e}\in J$ (the operation is well defined due to Lemma \ref{lemij=0}). 
	We say that a $(2e-1)$-dimensional subspace of $T$ is a \emph{hyperplane}.

Let $x$ and $y$ be elements of $J$ such that $x\neq 0, y\neq 0$. We observe that by Lemma \ref{lemjr}, $xu\in H$ if and only if $yu\in H$. This means that the statement of the following lemma does not depend on the choice of the generator $r$ of the ideal $J$.
    
	\begin{lemma}\label{lem:pu}
	Let $H$ be a hyperplane of $T$ and let $u\in V'$. Then the following statements hold.\\
		{\rm (1)} If $ru \in H$, then $|\{[u+h] : h \in H\}| = |H|/q$.\\
		{\rm (2)} If $ru \notin H$, then $|\{[u+h] : h \in H\}| = |H|$.\\
	\end{lemma}
	\begin{proof}
		 It suffices to show that, given $h \in H$, the equation
		\begin{equation}\label{eq:cond1}
			u+h = \sigma(u+f)    
		\end{equation}
		has exactly $q$ solutions (resp. a unique solution) in variables $\sigma \in K^\times$ and $f \in H$ whenever $ru \in H$ (resp. $ru \notin H$). 
		Equation (\ref{eq:cond1}) is equivalent to
		\begin{equation}\label{eq:cond2}
			(\sigma-1)u = h-\sigma f.    
		\end{equation}
		The right-hand side of Equation (\ref{eq:cond2}) is a $2e$-tuple with entries in $J$ and $u$ is a $2e$-tuple having at least one invertible entry. This implies that $\sigma-1$ belongs to  $J$, so $\sigma=1+z_jr$ for some $j\in\{1,2,\ldots,q\}$. Using Lemma \ref{lemij=0}, Equation (\ref{eq:cond2}) can be rewritten as
		$$
		z_jru = h - f
		$$
		for some $j\in\{1,2,\ldots,q\}$
		which is equivalent to
		\begin{equation}\label{eq:cond3}
			f =  h - z_jru.   
		\end{equation}
		Thus, if $ru \in H$, then the set of solutions $(\sigma,f)$ of Equation (\ref{eq:cond1}) is  
		$$
		\{(1+z_jr,h-z_jru) : j \in \{1,2,\ldots,q\}\}.
		$$
		Moreover, if $ru \notin H$, then $z_jru$ and, consequently, $h - z_jru$ do not belong to $H$ unless $j=1$, which means that $j=1$, so Equations (\ref{eq:cond3}) and (\ref{eq:cond1}) have a unique solution $(\sigma,f)$ given by $(1,h)$. This completes the proof.
	\end{proof}
	\begin{lemma}\label{lem:class}
		Let $H_1, H_2$ be hyperplanes of $T$ and let $u\in V'$ such that $ru \notin H_1 \cup H_2$. Then 
		$$\{[u+h] : h \in H_1\} = \{[u+f] : f \in H_2\}$$
		holds.
	\end{lemma}
	\begin{proof}
		In view of Lemma \ref{lem:pu}(2),
		it suffices to show that for any $g \in T \setminus H_1$ there exist $\sigma \in K^\times$ and $h \in H_1$ such that 
		$$
		u+h = \sigma(u+g).
		$$
		Note that if $j$ runs over $\{2,\ldots,q\}$ and $h$ runs over $H_1$, then $(h-z_jru)$ runs over $T \setminus H_1$. Thus, we have
		$$
		g = h - z_jru
		$$
		for some uniquely determined $j \in \{2,\ldots,q\}$ and $h \in H_1$. Put $\sigma = 1+z_jr$. Then, using Lemma \ref{lemij=0}, we have 
		$$
		\sigma(u+g) = (1+z_jr)(u+h-z_jru) = u+h,
		$$
		which completes the proof.
	\end{proof}
	
	Let $H$ be a hyperplane and $u$ be a $2e$-tuple over $K$ having an invertible entry, such that $ru \notin H$. Denote by $C(H,u)$ the set $\{[u+h] : h \in H\}$. We note that, for any $u_1\in V'$ such that $ru_1\notin H$ and $[u_1]\in C(H,u)$, the equality $C(H,u_1)=C(H,u)$ holds.
	We also use the \emph{Gaussian coefficient} $\qbin{n}{i}{q}$ to denote the number of $i$-dimensional subspaces of an $n$-dimensional vector space over $\F_q$.

	\begin{proposition}\label{prop:CHu}
	Let $H_1, H_2$ be hyperplanes of $T$ and let $u_1\in V'$, $u_2\in V'$ such that $ru_1\notin H_1, ru_2\notin H_2$. If the sets $C(H_1,u_1)$ and $C(H_2,u_2)$ have a vertex in common, then they coincide.
	\end{proposition}
	\begin{proof}
		Let the two sets $C(H_1,u_1)$ and $C(H_2,u_2)$ contain a common vertex $[u]$. This means there exist $h_1 \in H_1$ and $h_2 \in H_2$ such that $[u] = [u_1+h_1] = [u_2+h_2]$. That is, there exist $\sigma_1,\sigma_2 \in K^\times$ such that $\sigma_1u = u_1+h_1$ and $\sigma_2u = u_2+h_2$. This implies $u_1 = \sigma_1u - h_1$ and $u_2 = \sigma_2u - h_2$.
		Thus, we have
		\begin{align*}
			C(H_1,u_1) &= \{[u_1+h] : h \in H_1\}\\ 
			&= \{[\sigma_1u - h_1+h] : h \in H_1\}\\ 
			&= \{[\sigma_1u + h] : h \in H_1\}\\
			&= \{[u + \sigma_1^{-1}h] : h \in H_1\}\\
			&= \{[u + h] : h \in \sigma_1^{-1}H_1\}. 
		\end{align*}
		Note that $\sigma_1^{-1}H_1$ is a hyperplane as it is a subgroup of order $q^{2e-1}$ in $T$. Also, $ru$ does not belong to $\sigma_1^{-1}H_1$. Indeed, suppose $ru$ belongs to $\sigma_1^{-1}H_1$. Then there exists $f \in H_1$ such that
		$ru = \sigma_1^{-1}f$ and $f = \sigma_1ru$. The condition $u_1 = \sigma_1u - h_1$ then implies $ru_1 = \sigma_1ru = f$, that is, $ru_1$ belongs to $H_1$, a contradiction.
		Thus, we conclude that 
		$$
		C(H_1,u_1) = C(\sigma_1^{-1}H_1,u).
		$$
		In a similar way, we have that $\sigma_2^{-1}H_2$ is a hyperplane, $ru$ does not belong to $\sigma_2^{-1}H_2$ and 
		$$
		C(H_2,u_2) = C(\sigma_2^{-1}H_2,u).
		$$
		By Lemma \ref{lem:class}, we have $C(\sigma_1^{-1}H_1,u) = C(\sigma_2^{-1}H_2,u)$, which finally means $C(H_1,u_1) = C(H_2,u_2)$. This completes the proof. 
	\end{proof}
	\begin{corollary}\label{cor:Partition}
There exists a uniquely determined partition of the vertex sets of $X(2e,K)$ and $Y(2e,K)$ into classes $C(H,u)$. Moreover, the size of each class is $q^{2e-1}$ and the number of such classes in the partition is the Gaussian coefficient $\qbin{2e}{2e-1}{q} = \frac{q^{2e}-1}{q-1}$.
	\end{corollary}
	\begin{proof}
		This follows from Proposition \ref{prop:CHu} and Proposition \ref{prop1}.
	\end{proof}
	We denote the partition from Corollary \ref{cor:Partition} by $\Pi(2e,K)$.

\subsection{The numbers of common neighbours of two vertices in the graph $X(2e,K)$}
In this section, we determine the numbers of common neighbours of two vertices in the graph $X(2e,K)$.
	\begin{proposition}\label{propmain1}
	Let $C(H,u)$ be a class of the partition $\Pi(2e,K)$, where $H$ is a hyperplane of $T$ and $u\in V'$ such that $ru\notin H$. Then any two distinct vertices from $C(H,u)$ non-adjacent in $\overline{X(2e,K)}$ have $\dfrac{(q^{2e-2}-1)q^{2e-2}}{q-1}$ common neighbours in $\overline{X(2e,K)}$, and any two vertices from $C(H,u)$ adjacent in $\overline{X(2e,K)}$ have $\dfrac{(q^{2e-2}-1)q^{2e-2}}{q-1} - 2$ common neighbours in $\overline{X(2e,K)}$.     
	\end{proposition}
	\begin{proof}
		Let $[a],[b]$ be two distinct vertices from $C(H,u)$.
		Without loss of generality, we may assume $a = u$ and $b = u+h$ for some $h \in H$ where $h$ is not identically $0$.
		
		Let $[x]$ be a vertex of $\overline{X(2e,K)}$ that is a common neigbour of $[a]$ and $[b]$. Then we have
		\begin{equation}\label{eq:condition1}
			(a_1x_{e+1}-a_{e+1}x_1) +  (a_2x_{e+2}-a_{e+2}x_2) + \dots + (a_ex_{2e}-a_{2e}x_e) = 0  
		\end{equation}
		and
		\begin{equation}\label{eq:condition2}
			(b_1x_{e+1}-b_{e+1}x_1) +  (b_2x_{e+2}-b_{e+2}x_2) + \dots + (b_ex_{2e}-b_{2e}x_e) = 0.
		\end{equation}
		Since $b = a + h$, Equation (\ref{eq:condition2}) implies
		\begin{align}\label{eq:condition7}
			&(a_1x_{e+1}-a_{e+1}x_1) +  (a_2x_{e+2}-a_{e+2}x_2) + \dots + (a_ex_{2e}-a_{2e}x_e) + \notag\\
			& (h_1x_{e+1}-h_{e+1}x_1) +  (h_2x_{e+2}-h_{e+2}x_2) + \dots + (h_ex_{2e}-h_{2e}x_e) = 0
		\end{align}
		and, consequently,
		\begin{align}\label{eq:condition8}
			(h_1x_{e+1}-h_{e+1}x_1) +  (h_2x_{e+2}-h_{e+2}x_2) + \dots + (h_ex_{2e}-h_{2e}x_e) = 0
		\end{align}
		
		A vertex $[x]$ is a common neighbour of $[a]$ and $[b]$ if and only if $x$ satisfies the system of Equations (\ref{eq:condition1}) and (\ref{eq:condition8}). Thus, it suffices to count such $2e$-tuples $(x_1,x_2,\ldots,x_{2e})$.
		Using Lemma \ref{lemjr}, we assume that for every $j \in \{1,2,\ldots,2e\}$, $a_j = c_j+d_jr$ and $x_j = s_j+t_jr$, where $c_j,d_j,s_j,t_j$ are from $\{z_1,z_2,\ldots,z_q\}$, $c_j,d_j$ are constants and $s_j,t_j$ are variables.
		Equation (\ref{eq:condition1}) is then equivalent to
		\begin{align}\label{eq:condition9}
			& (c_1+d_1r)(s_{e+1}+t_{e+1}r) - (c_{e+1}+d_{e+1}r)(s_1+d_1r) +\notag\\
			& (c_2+d_2r)(s_{e+2}+t_{e+2}r) - (c_{e+2}+d_{e+2}r)(s_2+d_2r) + \dots +\notag\\
			& (c_e+d_er)(s_{2e}+t_{2e}r) - (c_{2e}+d_{2e}r)(s_e+d_er) = 0,
		\end{align}
		and using Lemma \ref{lemij=0}, Equation \eqref{eq:condition9} is equivalent to 
		\begin{align}\label{eq:condition10}
			& (c_1s_{e+1} + (c_1t_{e+1}+d_1s_{e+1})r) - (c_{e+1}s_1 + (c_{e+1}t_1+d_{e+1}s_1)r) +\notag\\
			& (c_2s_{e+2} + (c_2t_{e+2}+d_2s_{e+2})r) - (c_{e+2}s_2 + (c_{e+2}t_2+d_{e+2}s_2)r) + \dots +\notag\\
			& (c_es_{2e} + (c_et_{2e}+d_es_{2e})r) - (c_{2e}s_e + (c_{2e}t_e+d_{2e}s_e)r) = 0,
		\end{align}
		which we rewrite as
		\begin{align}\label{eq:group}
			& (c_1s_{e+1} - c_{e+1}s_1)  + (c_2s_{e+2}-c_{e+2}s_2) + \cdots + (c_es_{2e}-c_{2e}s_e) + \notag \\
			&\big((c_1t_{e+1}+d_1s_{e+1}) - (c_{e+1}t_1+d_{e+1}s_1)  + (c_2t_{e+2}+d_2s_{e+2})-(c_{e+2}t_2+d_{e+2}s_2) + \dots + \notag\\ & (c_et_{2e}+d_es_{2e})-(c_{2e}t_e+d_{2e}s_e)\big)r = 0.
		\end{align}
		Let 
		\begin{align}\label{eqfirst}
			(c_1s_{e+1} - c_{e+1}s_1)  + (c_2s_{e+2}-c_{e+2}s_2) + \cdots + (c_es_{2e}-c_{2e}s_e) = z_{j_1}+z_{j_2}r
		\end{align} for some $j_1,j_2\in\{1,2,\ldots,q\}$. Then, we deduce that Equation \ref{eq:group} is equivalent to the system of the following two linear equations modulo $J$:
		\begin{align}\label{eq:condition11}
			(c_1s_{e+1} - c_{e+1}s_1)  + (c_2s_{e+2}-c_{e+2}s_2) + \cdots + (c_es_{2e}-c_{2e}s_e) \equiv  0 \pmod J
		\end{align}
		and
		\begin{align}\label{eq:condition12}
			z_{j_2} + \bigg(\big((c_1t_{e+1}+d_1s_{e+1}) - (c_{e+1}t_1+d_{e+1}s_1) \big) + & \big((c_2t_{e+2}+d_2s_{e+2})-(c_{e+2}t_2+d_{e+2}s_2)\big) + \dots + \notag\\ & \big((c_et_{2e}+d_es_{2e})-(c_{2e}t_e+d_{2e}s_e)\big)\bigg) \equiv  0 \pmod J.
		\end{align}
		Indeed, let
		\begin{align*} E:=\big((c_1t_{e+1}+d_1s_{e+1}) - (c_{e+1}t_1+d_{e+1}s_1)\big)  + & \big((c_2t_{e+2}+d_2s_{e+2})-(c_{e+2}t_2+d_{e+2}s_2) \big)+ \dots + \notag\\ & \big((c_et_{2e}+d_es_{2e})-(c_{2e}t_e+d_{2e}s_e)\big),
		\end{align*}
		and let 
		\begin{align}\label{eqE}
			z_{j_2}+E=z_{j_3}+z_{j_4}r
		\end{align}
		for some $j_3,j_4\in\{1,2,\ldots,q\}$. Let Equation \eqref{eq:group} hold. Using Equation \eqref{eqfirst}, we obtain that  
		$(z_{j_1}+z_{j_2}r)+Er=0$, and using Lemma \ref{lemij=0} and Equation \eqref{eqE}, this implies that $z_{j_1}+z_{j_3}r=0$. So, Lemma \ref{lempart2} implies $z_{j_1}=z_{j_3}=0$, and this implies Equations \eqref{eq:condition11} and \eqref{eq:condition12}. Conversely, Equations \eqref{eq:condition11} and \eqref{eq:condition12} imply $z_{j_1}=z_{j_3}=0$, so the left hand side of Equation \eqref{eq:group} yields $(z_{j_1}+z_{j_2}r)+Er=(z_{j_2}+E)r=z_{j_3}r=0$. Now, we find the number of solutions of Equations \eqref{eq:condition8}, \eqref{eq:condition11} and \eqref{eq:condition12}.
		Since the $2e$-tuple $a$ has an invertible entry, there exists $i \in \{1,2,\ldots,2e\}$ such that $c_i \notin J$. Assume that $i$ belongs to $\{1,2,\ldots,e\}$ (otherwise, the proof is similar). Then Equation (\ref{eq:condition12}) is equivalent to
		\begin{align}\label{eq:condition13}
			t_{e+i} \equiv 
			&~c_{i}^{-1}\bigg(-z_{j_2}+
			\big((c_{e+1}t_1+d_{e+1}s_1) - (c_1t_{e+1}+d_1s_{e+1})\big) +\\
			&~\big((c_{e+2}t_2+d_{e+2}s_2) - (c_2t_{e+2}+d_2s_{e+2})\big)+  \dots + \notag\\
			&~\big((c_{e+i-1}t_{i-1}+d_{e+i-1}s_{i-1}) - (c_{i-1}t_{e+i-1}+d_{i-1}s_{e+i-1})\big) + \notag\\
			&~\big((c_{e+i}t_i+d_{e+i}s_i) - d_is_{e+i}\big) + \notag\\
			&~\big((c_{e+i+1}t_{i+1}+d_{e+i+1}s_{i+1}) - (c_{i+1}t_{e+i+1}+d_{i+1}s_{e+i+1})\big) + \dots + \notag\\
			&~\big((c_{2e}t_e+d_{2e}s_e) - (c_et_{2e}+d_es_{2e})\big)\bigg) \pmod J   
		\end{align}
		
		On the other hand, for every $j \in \{1,2,\ldots,2e\}$, let $h_j = w_jr$, where $w_j$ belongs to $\{z_1,z_2,\ldots,z_q\}$ (this is due to Lemma \ref{lempart1}). Then Equation (\ref{eq:condition8}) is equivalent to 
		\begin{align}\label{eq:condition14}
			\big(w_1r(s_{e+1}+t_{e+1}r)-w_{e+1}r(s_1+t_1r)\big) + 
			& \big(w_2r(s_{e+2}+t_{e+2}r)-w_{e+2}r(s_2+t_2r)\big) + \dots + \notag\\ & \big(w_er(s_{2e}+t_{2e}r)-w_{2e}r(s_e+t_er)\big) = 0,    
		\end{align}
		which is equivalent to
		\begin{align}\label{eq:condition15}
			(w_1s_{e+1}-w_{e+1}s_1) + 
			(w_2s_{e+2}-w_{e+2}s_2) + \dots + (w_es_{2e}-w_{2e}s_e) \equiv 0 \pmod J.    
		\end{align}
		Thus, the system of Equations (\ref{eq:condition1}) and (\ref{eq:condition8}) we need to solve is equivalent to the system of Equations (\ref{eq:condition11}), (\ref{eq:condition13}) and (\ref{eq:condition15}). 
		
		Let us show that the system of Equations (\ref{eq:condition11}) and (\ref{eq:condition15}) has rank $2e-2$. Since the $2e$-tuple $a$ has an invertible entry and the $2e$-tuple $h$ is not identically $0$, each of Equations (\ref{eq:condition11}) and (\ref{eq:condition15}) is non-trivial. Let us show that Equations (\ref{eq:condition11}) and (\ref{eq:condition15}) are linearly independent. Suppose to the contrary they are dependent, that is, there exists $\delta \in \{z_2,\ldots,z_q\}$ such that, for every $j \in \{1,2,\ldots,2e\}$, we have 
		\begin{equation}\label{eq:condition16}
			c_j \equiv \delta w_j \pmod J.    
		\end{equation}
		This implies
		$$ra = (c_1r,c_2r,\ldots,c_{2e}r) = \delta(w_1r,w_2r,\ldots,w_{2e}r) = \delta h$$
	belongs to $H$,	a contradiction. Thus, there exist $\ell, k$ such that $1 \le \ell < k \le 2e$ holds and the system of Equations (\ref{eq:condition11}) and (\ref{eq:condition15}) is equivalent to the system of equations
		\begin{equation}\label{eq:condition17}
			s_{\ell} \equiv M_{\ell}(s_1,s_2,\ldots,s_{\ell-1},s_{\ell+1},\ldots,s_{k-1},s_{k+1},\ldots,s_{2e}) \pmod J    
		\end{equation}
		and
		\begin{equation}\label{eq:condition18}
			s_{k} \equiv M_k(s_1,s_2,\ldots,s_{\ell-1},s_{\ell+1},\ldots,s_{k-1},s_{k+1},\ldots,s_{2e}) \pmod J,    
		\end{equation}
		where $M_{\ell}(s_1,s_2,\ldots,s_{\ell-1},s_{\ell+1},\ldots,s_{k-1},s_{k+1},\ldots,s_{2e})$ and $M_k(s_1,s_2,\ldots,s_{\ell-1},s_{\ell+1},\ldots,s_{k-1},s_{k+1},\ldots,s_{2e})$
		are linear combinations of the $2e-2$ variables $s_1,s_2,\ldots,s_{\ell-1},s_{\ell+1},\ldots,s_{k-1},s_{k+1},\ldots,s_{2e}$. Thus, the system of Equations (\ref{eq:condition11}), (\ref{eq:condition13}) and (\ref{eq:condition15}) is equivalent to the system of Equations (\ref{eq:condition17}), (\ref{eq:condition18}) and (\ref{eq:condition13}).
		Note that once the values of $s_1,s_2,\ldots,s_{2e}$ (and therefore the value of $z_{j_2}$ by Equation \eqref{eqfirst}) and $t_1,t_2,\ldots,t_{e+i-1},t_{e+i+1},\ldots,t_{2e}$ are determined, then the value of $t_{e+i}$ is also determined according to Equation (\ref{eq:condition13}). 
		
		Now we are ready to count the number of solutions of the system of Equations (\ref{eq:condition17}), (\ref{eq:condition18}) and (\ref{eq:condition13}).
		The system of Equations (\ref{eq:condition17}), (\ref{eq:condition18}) has $q^{2e-2}$ solutions, but the zero solution must be excluded as the $2e$-tuple $x$ has an invertible entry. Thus, $q^{2e-2}-1$ solutions of Equations (\ref{eq:condition17}) and (\ref{eq:condition18}) satisfy our requirements. Moreover, for every such a solution there are $q^{2e-1}$ solutions of Equation (\ref{eq:condition13}) as the variables  $t_1,t_2,\ldots,t_{e+i-1},t_{e+i+1},\ldots,t_{2e}$ are independent. Thus, the total number of solutions of the system of Equations (\ref{eq:condition17}), (\ref{eq:condition18}) and (\ref{eq:condition13}) is
		$$
		(q^{2e-2}-1)q^{2e-1}.
		$$
		Finally, we conclude that if the vertices $[a]$ and $[a+h]$ are non-adjacent, then they have 
		$$
		\dfrac{(q^{2e-2}-1)q^{2e-1}}{q^2-q} = \dfrac{(q^{2e-2}-1)q^{2e-2}}{q-1}
		$$
		common neighbours in $\overline{X(2e,K)}$.
		Also, if the vertices $[a]$ and $[a+h]$ are adjacent, then the $2e$-tuples from $[a]$ and $[a+h]$ were counted among the $$
		(q^{2e-2}-1)q^{2e-1}
		$$ solutions of the system of Equations (\ref{eq:condition17}), (\ref{eq:condition18}) and (\ref{eq:condition13}). Thus, the number of common neighbours of the vertices $[a]$ and $[a+h]$ in $\overline{X(2e,K)}$ is
		$$
		\dfrac{(q^{2e-2}-1)q^{2e-2}}{q-1} - 2,
		$$
		which completes the proof.
	\end{proof}

	\begin{corollary}\label{cor:XTwoVerticesFromTheSameClass}
	Let $C(H,u)$ be a class of the partition $\Pi(2e,K)$, where $H$ is a hyperplane of $T$ and $u\in V'$ such that $ru\notin H$. Then any two distinct vertices from $C(H,u)$ have $q^{4e-2}+q^{4e-3}-q^{4e-4}-q^{2e-2}$ common neighbours in $X(2e,K)$. 
	\end{corollary}

	A class $C(H,u)$ of the partition $\Pi(2e,K)$, being the set of vertices $\{[u+h] : h \in H\}$ can be naturally identified with the set of $2e$-tuples $\{\sigma(u+h) : \sigma \in K^\times, h \in H\}$ of size $(q^2-q)q^{2e-1} = (q-1)q^{2e}$.
	
	\begin{lemma}\label{lem:cosetsOfT}
		Let $C(H,u)$ be a class of the partition $\Pi(2e,K)$, where $H$ is a hyperplane of $T$ and $u\in V'$ such that $ru\notin H$. Then the corresponding set of $2e$-tuples $\{\sigma(u+h) : \sigma \in K^\times, h \in H\}$ includes the coset $u + T$ and can moreover be partitioned into $q-1$ cosets of $T$, where a set of representatives of these cosets is $\{z_ju : j \in \{2,\ldots,q\}\}$. 
	\end{lemma}
	\begin{proof}
		We recall that $J=\langle r\rangle$. Let $\sigma\in K^\times$ and $h\in H$. We consider the difference $\sigma(u+h) - z_ju$ for some $j\in\{2,\ldots,q\}$. Using Lemma \ref{lemjr}, we assume that $\sigma = s + tr$ where $s,t \in \{z_1,z_2,\ldots,z_q\}$ and $s\neq z_1$ by Lemma \ref{lempart1}. Then, using Lemma \ref{lemij=0} we have 
		\begin{equation}\label{eq:condition19}
			\sigma(u+h) - z_ju = (s+tr)(u+h) - z_ju = su + tru + sh - z_ju = (s-z_j)u + \underbrace{tru}_{\substack{\notin H\\\text{ if }t\neq 0}}+\underbrace{sh}_{\in H}.    
		\end{equation}
				Put $s = z_j$ in Equation (\ref{eq:condition19}). Then we get
		\begin{equation}\label{eq:condition20}
			\sigma(u+h) - z_ju= \underbrace{tru}_{\substack{\notin H\\\text{ if }t\neq 0}}+\underbrace{z_jh}_{\in H}.    
		\end{equation}
		In Equation (\ref{eq:condition20}), if $t$ runs over $
		\{z_1,z_2,\ldots,z_q\}$ and $h$ runs over $H$, then
		$(tru + z_jh)$ runs over $T$. Thus, the set of $2e$-tuples $\{\sigma(u+h) : \sigma \in K^\times, h \in H\}$ includes the coset $z_ju + T$, and if $1\in z_i+J$ for some $i\in\{2,\ldots,q\}$, then $z_iu+T=u+T$. Moreover, this set of $2e$-tuples can be partitioned into $q-1$ cosets of $T$, where a set of representatives of these cosets is $\{z_ju : j \in \{2,\ldots,q\}\}$. This completes the proof.    
	\end{proof}
	
	\begin{proposition}\label{propmain2}
	Let $H_1,H_2$ be hyperplanes of $T$, and let $a\in V'$, $b\in V'$ with $ra\notin H_1,rb\notin H_2$, such that $C(H_1,a)$ and $C(H_2,b)$ are distinct classes of the partition $\Pi(2e,K)$. Then if the vertices $[a]$ and $[b]$ are non-adjacent in $\overline{X(2e,K)}$, then they have $\dfrac{(q^{2e-2}-1)q^{2e-3}}{q-1}$ common neighbours in $\overline{X(2e,K)}$, and if the vertices $[a]$ and $[b]$ are adjacent in $\overline{X(2e,K)}$, then they have $\dfrac{(q^{2e-2}-1)q^{2e-3}}{q-1} - 2$ common neighbours in $\overline{X(2e,K)}$.   
	\end{proposition}
	\begin{proof}
		Since $[b]$ does not belong to $C(H_1,a)$, we conclude that there exists no  $\sigma \in K^\times$ and $h \in H_1$ such that $b = \sigma(a+h)$ holds. 
		
		Let $[x]$ be a common neighbour of $[a]$ and $[b]$. Then we have a system of the equations
		\begin{equation}\label{eq:condition21}
			(a_1x_{e+1}-a_{e+1}x_1) +  (a_2x_{e+2}-a_{e+2}x_2) + \dots + (a_ex_{2e}-a_{2e}x_e) = 0  
		\end{equation}
		and
		\begin{equation}\label{eq:condition22}
			(b_1x_{e+1}-b_{e+1}x_1) +  (b_2x_{e+2}-b_{e+2}x_2) + \dots + (b_ex_{2e}-b_{2e}x_e) = 0.
		\end{equation}
		We consider the ordering of variables $x_{e+1},x_{e+2},\ldots,x_{2e},x_1,\ldots,x_e$ and write down the matrix of coefficients:
		$$
		\left(
		\begin{matrix}
			a_1 & a_2 & \dots & a_{2e}\\
			b_1 & b_2 & \dots & b_{2e}
		\end{matrix}
		\right)
		$$
		Note that there exist $i \in \{1,2,\ldots,2e\}$ such that $a_i$ is invertible. Then the matrix of coefficients can be written as
		\begin{align*}
		\left(
		\begin{matrix}
			a_1 & \dots & a_i &  \dots & a_{2e}\\
			b_1 & \dots & b_i &  \dots & b_{2e}
		\end{matrix}
		\right),
		\end{align*}
		which is equivalent to
		\begin{align*}
		\left(
		\begin{matrix}
			a_1a_i^{-1} & \dots & 1 &  \dots & a_{2e}a_i^{-1}\\
			b_1 & \dots & b_i & \dots  & b_{2e}
		\end{matrix}
		\right)
		\end{align*}
		and, consequently,
		\begin{align}\label{mat}
		\left(
		\begin{matrix}
			a_1a_i^{-1} & \dots & 1 & \dots &  a_{2e}a_i^{-1}\\
			b_1-a_1a_i^{-1}b_i & \dots & 0 & \dots & b_{2e}-a_{2e}a_i^{-1}b_i
		\end{matrix}
		\right).
		\end{align}
	Now, let us show that there exists $\gamma\in \{1,2,\ldots,2e\}\setminus\{i\}$ such that $b_{\gamma}-a_{\gamma}a_i^{-1}b_i$ is invertible. To begin with, we find that there exists $j \in \{1,2,\ldots,2e\} \setminus \{i\}$ such that $b_j$ is invertible (otherwise, $a_i$ and $b_i$ are the only invertible entries of $a$ and $b$ and there exist $\sigma \in K^\times$ and $h \in T$ such that $b = \sigma(a+h)$, a contradiction by Lemma \ref{lem:cosetsOfT}). Suppose to the contrary that $b_{\gamma}-a_{\gamma}a_i^{-1}b_i$ is a zero divisor for all $\gamma \in  \{1,2,\ldots,2e\}\setminus\{i\}$,
		that is, 
		$$b_{\gamma}-a_{\gamma}a_i^{-1}b_i = t_{\gamma}r$$ for some $t_\gamma\in\{z_1,z_2,\ldots,z_q\}$.  In particular, we have 
		\begin{equation}\label{eq:condition23}
			b_j-a_ja_i^{-1}b_i = t_jr,
		\end{equation}
		which implies that $a_j$ and $b_i$ are also invertible (otherwise, the element on left hand side of the equality in Equation (\ref{eq:condition23}) is invertible and equal to a zero divisor, a contradiction). Thus, without loss of generality, we may assume $a_i = b_i = 1$ (we can multiply the original two equations by $a_i^{-1}$ and $b_i^{-1}$). 	We also deduce that for any $\gamma \in \{1,2,\ldots,2e\}$ the elements $a_{\gamma}$ and $b_{\gamma}$ are both invertible or both zero divisors. Equation (\ref{eq:condition23}) then becomes
		\begin{equation}\label{eq:condition24}
			b_j-a_j = t_jr.
		\end{equation}
		We observe that as $j$ varies over $\{1,2,\ldots,2e\}\setminus\{i\}$ such that $b_j$ is invertible, Equation \eqref{eq:condition24} holds. This means that $b-a \in T$, that is, $a$ and $b$ belong to the same coset of $T$. Then by Lemma \ref{lem:cosetsOfT}, $[a]$ and $[b]$ belong to the same class of the partition $\Pi(2e,K)$, which contradicts the choice of $[a]$ and $[b]$.
		
		Thus there exist $\ell, k \in \{1,2,\ldots,2e\}$ such that $1 \le \ell < k \le 2e$ and the system of Equations (\ref{eq:condition21}) and (\ref{eq:condition22}) is equivalent to the system of equations
		\begin{equation}\label{eq:condition25}
			x_{\ell} = M_{\ell}(x_1,x_2,\ldots,x_{\ell-1},x_{\ell+1},\ldots,x_{k-1},x_{k+1},\ldots,x_{2e})    
		\end{equation}
		and
		\begin{equation}\label{eq:condition26}
			x_{k} = M_k(x_1,x_2,\ldots,x_{\ell-1},x_{\ell+1},\ldots,x_{k-1},x_{k+1},\ldots,x_{2e}),    
		\end{equation}
		where $M_{\ell}(x_1,x_2,\ldots,x_{\ell-1},x_{\ell+1},\ldots,x_{k-1},x_{k+1},\ldots,x_{2e})$ and $M_k(x_1,x_2,\ldots,x_{\ell-1},x_{\ell+1},\ldots,x_{k-1},x_{k+1},\ldots,$\\$x_{2e})$
		are linear combinations of the $2e-2$ variables $x_1,x_2,\ldots,x_{\ell-1},x_{\ell+1},\ldots,x_{k-1},x_{k+1},\ldots,x_{2e}$.
		
		Now we are ready to count the number of solutions of the system of Equations (\ref{eq:condition25})  and (\ref{eq:condition26}) in the variables $x_1,x_2,\ldots,x_{2e},y$ and $z$.
		It has $(q^2)^{2e-2}$ solutions, but the solutions $x$ such that all the entries are zero divisors must be excluded as the $2e$-tuple $x$ has an invertible entry. Thus, $(q^2)^{2e-2}-q^{2e-2}$ solutions of Equations (\ref{eq:condition25}) and (\ref{eq:condition26}) satisfy our requirements.  
		Finally, we conclude that if the vertices $[a]$ and $[b]$ are non-adjacent, then they have 
		$$
		\dfrac{(q^{2e-2}-1)q^{2e-2}}{q^2-q} = \dfrac{(q^{2e-2}-1)q^{2e-3}}{q-1}
		$$
		common neighbours in $\overline{X(2e,K)}$.
		Also, if the vertices $[a]$ and $[b]$ are adjacent, then the $2e$-tuples from $[a]$ and $[b]$ were counted among the $$
		(q^{2e-2}-1)q^{2e-2}
		$$ solutions of the system of Equations (\ref{eq:condition25}) and (\ref{eq:condition26}). Thus, the number of common neighbours of the vertices $[a]$ and $[b]$ in $\overline{X(2e,K)}$ is
		$$
		\dfrac{(q^{2e-2}-1)q^{2e-3}}{q-1} - 2,
		$$
		and the proof is complete.
	\end{proof}
	
	\begin{corollary}\label{cor:XTwoVertices FromDifferentClasses}
	Let $H_1,H_2$ be hyperplanes of $T$, and let $a,b\in V'$ with $ra\notin H_1,rb\notin H_2$, such that $C(H_1,a)$ and $C(H_2,b)$ are distinct classes of the partition $\Pi(2e,K)$. Then the vertices $[a]$ and $[b]$ have $q^{4e-2}+q^{4e-3}-q^{4e-4}-q^{4e-5}-q^{2e-2}+q^{2e-3}$ common neighbours in $X(2e,K)$.      
	\end{corollary}
	
\subsection{The numbers of common neighbours of two vertices in the graph $Y(2e,K)$}
In this section, we determine the numbers of common neighbours of two vertices in the graph $Y(2e,K)$.

\begin{proposition}\label{prop:YTwoVerticesFromTheSameClass}
		Let $C(H,u)$ be a class of the partition $\Pi(2e,K)$, where $H$ is a hyperplane of $T$ and $u\in V'$ such that $ru\notin H$. Then any two distinct vertices from $C(H,u)$ have $q^{4e-3}-q^{4e-4}-q^{2e-2}$ common neighbours in $Y(2e,K)$.
	\end{proposition}
	\begin{proof}
		We proceed similarly as in the proof of Proposition \ref{propmain1}. Let $[a],[b]$ be two distinct vertices from $C(H,u)$.	Without loss of generality, we may assume $a = u$ and $b = u+h$ for some $h \in H$ where $h$ is not identically $0$. Let $[x]$ be a vertex of $\overline{X(2e,K)}$ that is a common neigbour of $[a]$ and $[b]$. Analogous to Equations \eqref{eq:condition1} and \eqref{eq:condition7}, we have
		\begin{equation}\label{yconda}
			(a_1x_{e+1}-a_{e+1}x_1) +  (a_2x_{e+2}-a_{e+2}x_2) + \dots + (a_ex_{2e}-a_{2e}x_e) = yr\text{ for some }y\in\{z_2,\ldots,z_q\}  
		\end{equation}
		and
		\begin{align}\label{eqcond1}
			&(a_1x_{e+1}-a_{e+1}x_1) +  (a_2x_{e+2}-a_{e+2}x_2) + \dots + (a_ex_{2e}-a_{2e}x_e) + \notag\\
			& (h_1x_{e+1}-h_{e+1}x_1) +  (h_2x_{e+2}-h_{e+2}x_2) + \dots + (h_ex_{2e}-h_{2e}x_
			e) = zr\text{ for some }y\in\{z_2,\ldots,z_q\} ,
		\end{align}
		and Equation \eqref{eqcond1} implies that
		\begin{align}\label{ycondb}
			(h_1x_{e+1}-h_{e+1}x_1) +  (h_2x_{e+2}-h_{e+2}x_2) + \dots + (h_ex_{2e}-h_{2e}x_e) = (z-y)r
		\end{align}
		Our task is to solve Equations \eqref{yconda} and \eqref{ycondb}. Using Lemma \ref{lemjr}, we assume that for every $j \in \{1,2,\ldots,2e\}$, $a_j = c_j+d_jr$ and $x_j = s_j+t_jr$, where $c_j,d_j,s_j,t_j$ are from $\{z_1,z_2,\ldots,z_q\}$, $c_j,d_j$ are constants and $s_j,t_j$ are variables.
		Then Equation \eqref{yconda} yields
		\begin{align}\label{sud}
			& (c_1+d_1r)(s_{e+1}+t_{e+1}r) - (c_{e+1}+d_{e+1}r)(s_1+d_1r) +\notag\\
			& (c_2+d_2r)(s_{e+2}+t_{e+2}r) - (c_{e+2}+d_{e+2}r)(s_2+d_2r) + \dots +\notag\\
			& (c_e+d_er)(s_{2e}+t_{2e}r) - (c_{2e}+d_{2e}r)(s_e+d_er) = yr,
		\end{align}
		and using Lemma \ref{lemij=0}, Equation \eqref{sud} yields
	\begin{align}\label{ycondaa}
			& (c_1s_{e+1} + (c_1t_{e+1}+d_1s_{e+1})r) - (c_{e+1}s_1 + (c_{e+1}t_1+d_{e+1}s_1)r) +\notag\\
			& (c_2s_{e+2} + (c_2t_{e+2}+d_2s_{e+2})r) - (c_{e+2}s_2 + (c_{e+2}t_2+d_{e+2}s_2)r) + \dots +\notag\\
			& (c_es_{2e} + (c_et_{2e}+d_es_{2e})r) - (c_{2e}s_e + (c_{2e}t_e+d_{2e}s_e)r) = yr.
		\end{align}
		Now, let
		\begin{align}\label{eqcond5}
		(c_1s_{e+1} - c_{e+1}s_1)  + (c_2s_{e+2}-c_{e+2}s_2) + \cdots + (c_es_{2e}-c_{2e}s_e) = z_{j_1}+z_{j_2}r
	\end{align} for some $j_1,j_2\in\{1,2,\ldots,q\}$.
	We proceed as in the proof of Proposition \ref{propmain1}, where we concluded that Equation \eqref{eq:group} is equivalent to Equations \eqref{eq:condition11} and \eqref{eq:condition12}. Here, we deduce that Equation \eqref{ycondaa} is equivalent to the system of the following two linear equations modulo $J$:
		\begin{align}\label{ycondabreak1}
			(c_1s_{e+1} - c_{e+1}s_1)  + (c_2s_{e+2}-c_{e+2}s_2) + \cdots + (c_es_{2e}-c_{2e}s_e) \equiv  0 \pmod J
		\end{align}
		and
		\begin{align}\label{ycondabreak2}
		z_{j_2}+	\big((c_1t_{e+1}+d_1s_{e+1}) - (c_{e+1}t_1+d_{e+1}s_1)\big)  + & \big((c_2t_{e+2}+d_2s_{e+2}-c_{e+2}t_2+d_{e+2}s_2)\big) + \dots + \notag\\ & \big((c_et_{2e}+d_es_{2e})-(c_{2e}t_e+d_{2e}s_e)\big) \equiv  y \pmod J.
		\end{align}
		Now, we find the number of solutions of Equations \eqref{ycondb}, \eqref{ycondabreak1} and \eqref{ycondabreak2}.
		Since the $2e$-tuple $a$ has an invertible entry, there exists $i \in \{1,2,\ldots,2e\}$ such that $c_i \notin J$. Assume that $i$ belongs to $\{1,2,\ldots,e\}$ (otherwise, the proof is similar). Then Equation (\ref{ycondabreak2}) is equivalent to
		\begin{align}\label{ycondabreak2equiv}
			t_{e+i} \equiv 
			&~c_{i}^{-1}\bigg(y-z_{j_2}+\big((c_{e+1}t_1+d_{e+1}s_1) - (c_1t_{e+1}+d_1s_{e+1})\big) +\\
			&~\big((c_{e+2}t_2+d_{e+2}s_2) - (c_2t_{e+2}+d_2s_{e+2})\big)+ \dots + \notag\\
			&~\big((c_{e+i-1}t_{i-1}+d_{e+i-1}s_{i-1}) - (c_{i-1}t_{e+i-1}+d_{i-1}s_{e+i-1})\big) + \notag\\
			&~ \big((c_{e+i}t_i+d_{e+i}s_i) - d_is_{e+i}\big) + \notag\\
			&~\big((c_{e+i+1}t_{i+1}+d_{e+i+1}s_{i+1}) - (c_{i+1}t_{e+i+1}+d_{i+1}s_{e+i+1})\big) + \dots + \notag\\
			&~\big((c_{2e}t_e+d_{2e}s_e) - (c_et_{2e}+d_es_{2e})\big)\bigg) \pmod J    
		\end{align}
	On the other hand, for every $j \in \{1,2,\ldots,2e\}$, let $h_j = w_jr$, where $w_j$ belongs to $\{z_1,z_2,\ldots,z_q\}$ (this is due to Lemma \ref{lempart1}). Then Equation \eqref{ycondb} yields
		\begin{align}
			\big(w_1r(s_{e+1}+t_{e+1}r)-w_{e+1}r(s_1+t_1r)\big) + 
			& \big(w_2r(s_{e+2}+t_{e+2}r)-w_{e+2}r(s_2+t_2r)\big) + \dots + \notag\\ & \big(w_er(s_{2e}+t_{2e}r)-w_{2e}r(s_e+t_er)\big) = (z-y)r,    
		\end{align}
		which is equivalent to
		\begin{align}\label{ycondbbreak}
			(w_1s_{e+1}-w_{e+1}s_1) + 
			(w_2s_{e+2}-w_{e+2}s_2) + \dots + (w_es_{2e}-w_{2e}s_e) \equiv z-y \pmod J.    
		\end{align}
		Thus, the system of equations we need to solve is equivalent to the system of Equations \eqref{ycondabreak1}, \eqref{ycondabreak2equiv} and \eqref{ycondbbreak}. 
		
		Similar to the proof of Proposition \ref{propmain1}, we find that the system of Equations (\ref{ycondabreak1}) and (\ref{ycondbbreak}) has rank $2e-2$, as follows. Since the $2e$-tuple $a$ has an invertible entry and the $2e$-tuple $h$ is not identically $0$, each of Equations (\ref{ycondabreak1}) and (\ref{ycondbbreak}) is non-trivial. Let us show that Equations (\ref{ycondabreak1}) and (\ref{ycondbbreak}) are linearly independent. Suppose to the contrary they are dependent, that is, there exists $\delta \in \{z_2,\ldots,z_q\}$ such that, for every $j \in \{1,2,\ldots,2e\}$, we have 
		\begin{equation}\label{eqcond2}
			c_j \equiv \delta w_j \pmod J.    
		\end{equation}
		This implies that
		$ra$ belongs to $H$, a contradiction. Thus, there exist $\ell, k$ such that $1 \le \ell < k \le 2e$ holds and the system of Equations (\ref{ycondabreak1}) and (\ref{ycondbbreak}) is equivalent to the system of equations
		\begin{equation}\label{eqcond3}
			s_{\ell} \equiv M_{\ell}(s_1,s_2,\ldots,s_{\ell-1},s_{\ell+1},\ldots,s_{k-1},s_{k+1},\ldots,s_{2e},y,z) \pmod J    
		\end{equation}
		and
		\begin{equation}\label{eqcond4}
			s_{k} \equiv M_k(s_1,s_2\ldots,s_{\ell-1},s_{\ell+1},\ldots,s_{k-1},s_{k+1},\ldots,s_{2e},y,z) \pmod J,    
		\end{equation}
		where $M_{\ell}(s_1,s_2,\ldots,s_{\ell-1},s_{\ell+1},\ldots,s_{k-1},s_{k+1},\ldots,s_{2e},y,z)$ and $M_k(s_1,s_2,\ldots,s_{\ell-1},s_{\ell+1},\ldots,s_{k-1},s_{k+1},\ldots,$\\$s_{2e},y,z)$
		are linear combinations of the $2e$ variables $s_1,s_2,\ldots,s_{\ell-1},s_{\ell+1},\ldots,s_{k-1},s_{k+1},\ldots,s_{2e},y,z$. Thus, the system of Equations (\ref{ycondabreak1}), (\ref{ycondabreak2equiv}) and (\ref{ycondbbreak}) is equivalent to the system of Equations (\ref{ycondabreak2equiv}), (\ref{eqcond3}) and (\ref{eqcond4}).
		Note that once the values of $s_1,s_2,\ldots,s_{2e}$ (and therefore the value of $z_{j_2}$ by Equation \eqref{eqcond5}) and $t_1,t_2,\ldots,t_{e+i-1},t_{e+i+1},\ldots,t_{2e}$ are determined, then the value of $t_{e+i}$ is also determined according to Equation (\ref{ycondabreak2equiv}). 
		
		Now we are ready to count the number of solutions of the system of Equations (\ref{ycondabreak2equiv}), (\ref{eqcond3}) and (\ref{eqcond4}).
		Depending on the values of $y$ and $z$, we have the following two cases:
		\begin{itemize}
			\item Case 1: $y=z$, and
			\item Case 2: $y\neq z$.
		\end{itemize}
		Corresponding to Case 1, the system of Equations \eqref{eqcond3} and \eqref{eqcond4} in the variables $s_1,s_2,\ldots,s_{2e}$ has $q^{2e-2}$ solutions, but the zero solution must be excluded as the $2e$-tuple $x$ has an invertible entry. Thus, $q^{2e-2}-1$ solutions of Equations \eqref{eqcond3} and \eqref{eqcond4} satisfy our requirements. For every such a solution there are $q^{2e-1}$ solutions of Equation \eqref{ycondabreak2equiv} as the variables  $t_1,t_2,\ldots,t_{e+i-1},t_{e+i+1},\ldots,t_{2e}$ are independent. Moreover, since $y,z\in\{z_2,\ldots,z_q\}$, there are $(q-1)$ possibilities of $y=z$. Thus, the total number of solutions of the system of Equations \eqref{ycondabreak2equiv}, \eqref{eqcond3} and \eqref{eqcond4} for Case 1 is
		$$
		(q^{2e-2}-1)q^{2e-1}(q-1).
		$$
		Corresponding to Case 2, the system of Equations \eqref{eqcond3} and \eqref{eqcond4} in the variables $s_1,s_2,\ldots,s_{2e}$ has $q^{2e-2}$ solutions. For every such a solution there are $q^{2e-1}$ solutions of Equation \eqref{ycondabreak2equiv}. Moreover, there are $(q-1)^2-(q-1)=(q-1)(q-2)$ possibilities of $y\neq z$. Thus, the total number of solutions of the system of Equations \eqref{ycondabreak2equiv}, \eqref{eqcond3} and \eqref{eqcond4} for Case 2 is
		$$
		q^{2e-2}q^{2e-1}(q-1)(q-2).
		$$
		So, we have that corresponding to Case 1, the vertices $[a]$ and $[a+h]$ have 
		$$
		\dfrac{(q^{2e-2}-1)q^{2e-1}(q-1)}{q^2-q} = (q^{2e-2}-1)q^{2e-2}
		$$
		common neighbours in ${Y(2e,K)}$, and corresponding to Case 2, the vertices $[a]$ and $[a+h]$ have 
		$$
		\dfrac{q^{2e-2}q^{2e-1}(q-1)(q-2)}{q^2-q} = q^{2(2e-2)}(q-2)
		$$
		common neighbours in ${Y(2e,K)}$. Finally, we combine Cases 1 and 2 and conclude that the vertices $[a]$ and $[a+h]$ have $$(q^{2e-2}-1)q^{2e-2}+q^{2(2e-2)}(q-2)=q^{4e-3}-q^{4e-4}-q^{2e-2}$$ common neighbours in $Y(2e,K)$, and the proof is complete.
	\end{proof}
	\begin{proposition}\label{prop:YTwoVerticesFromDifferentClasses}
		Let $H_1,H_2$ be hyperplanes of $T$, and let $a\in V'$, $b\in V'$ with $ra\notin H_1,rb\notin H_2$, such that $C(H_1,a)$ and $C(H_2,b)$ are distinct classes of the partition $\Pi(2e,K)$. Then $[a]$ and $[b]$ have $q^{2e-3}(q-1)(q^{2e-2}-1)$ common neighbours in $Y(2e,K)$.
	\end{proposition}
	\begin{proof}
	The proof goes along similar lines as that of Proposition \ref{propmain2}. Since $[b]$ does not belong to $C(H_1,a)$, we conclude that there exists no  $\sigma \in K^\times$ and $h \in H_1$ such that $b = \sigma(a+h)$ holds. 
		
		Let $[x]$ be a common neighbour of $[a]$ and $[b]$. Then we have a system of the equations
		\begin{equation}\label{eq:condition21y}
			(a_1x_{e+1}-a_{e+1}x_1) +  (a_2x_{e+2}-a_{e+2}x_2) + \dots + (a_ex_{2e}-a_{2e}x_e) = yr 
		\end{equation}
		and
		\begin{equation}\label{eq:condition22y}
			(b_1x_{e+1}-b_{e+1}x_1) +  (b_2x_{e+2}-b_{e+2}x_2) + \dots + (b_ex_{2e}-b_{2e}x_e) = zr
		\end{equation}
		for some $y,z\in\{z_2,\ldots,z_q\}$.
		We consider the ordering of variables $x_{e+1},x_{e+2},\ldots,x_{2e},x_1,\ldots,x_e$ and write down the augmented matrix:
		$$
		\left(
		\begin{array}{cccc|c}
			a_1 & a_2 & \dots & a_{2e}& yr\\
			b_1 & b_2 & \dots & b_{2e}& zr
		\end{array}
		\right)
		$$
		Note that there exist $i \in \{1,2,\ldots,2e\}$ such that $a_i$ is invertible. Then the matrix of coefficients can be written as
		\begin{align*}
		\left(
		\begin{array}{ccccc|c}
			a_1 & \dots & a_i &  \dots & a_{2e}& yr\\
			b_1 & \dots & b_i &  \dots & b_{2e}& zr
		\end{array}
		\right),
		\end{align*}
		which is equivalent to
		\begin{align*}
		\left(
		\begin{array}{ccccc|c}
			a_1a_i^{-1} & \dots & 1 &  \dots & a_{2e}a_i^{-1} & y ra_i^{-1}\\
			b_1 & \dots & b_i & \dots  & b_{2e} & zr
		\end{array}
		\right),
		\end{align*}
		and, consequently,
		\begin{align}\label{matt}
		\left(
		\begin{array}{ccccc|c}
			a_1a_i^{-1} & \dots & 1 & \dots &  a_{2e}a_i^{-1}& yr a_i^{-1}\\
			b_1-a_1a_i^{-1}b_i & \dots & 0 & \dots & b_{2e}-a_{2e}a_i^{-1}b_i &zr-yra_i^{-1}b_i
		\end{array}
		\right).
		\end{align}
		Since the matrix of coefficients in \eqref{matt} equals the matrix in \eqref{mat} with the same conditions on the entries in $a$ and $b$, we conclude (as in the proof of Proposition \ref{propmain1}) that there exists $\gamma\in \{1,2,\ldots,2e\}\setminus\{i\}$ such that $b_{\gamma}-a_{\gamma}a_i^{-1}b_i$ is invertible.		
		Thus there exist $\ell, k \in \{1,2,\ldots,2e\}$ such that $1 \le \ell < k \le 2e$ and the system of Equations (\ref{eq:condition21y}) and (\ref{eq:condition22y}) is equivalent to the system of equations
		\begin{equation}\label{eq:condition25y}
			x_{\ell} = M_{\ell}(x_1,x_2,\ldots,x_{\ell-1},x_{\ell+1},\ldots,x_{k-1},x_{k+1},\ldots,x_{2e},y,z)    
		\end{equation}
		and
		\begin{equation}\label{eq:condition26y}
			x_{k} = M_k(x_1,x_2,\ldots,x_{\ell-1},x_{\ell+1},\ldots,x_{k-1},x_{k+1},\ldots,x_{2e},y,z),    
		\end{equation}
		where $M_{\ell}(x_1,x_2,\ldots,x_{\ell-1},x_{\ell+1},\ldots,x_{k-1},x_{k+1},\ldots,x_{2e},y,z)$ and $M_k(x_1,x_2,\ldots,x_{\ell-1},x_{\ell+1},\ldots,x_{k-1},x_{k+1},$\\$\ldots,x_{2e},y,z)$
		are linear combinations of the $2e$ variables $x_1,x_2,\ldots,x_{\ell-1},x_{\ell+1},\ldots,x_{k-1},x_{k+1},\ldots,x_{2e},y$, and $z$.
		Now we are ready to count the number of solutions of the system of Equations (\ref{eq:condition25y})  and (\ref{eq:condition26y}) in the variables $x_1,x_2,\ldots,x_{2e},y$ and $z$.
		It has $(q^2)^{2e-2}(q-1)^2$ solutions, but the solutions $x$ such that all the entries are zero divisors must be excluded as the $2e$-tuple $x$ has an invertible entry. Thus, $((q^2)^{2e-2}-q^{2e-2})(q-1)^2$ solutions of Equations (\ref{eq:condition25y}) and (\ref{eq:condition26y}) satisfy our requirements, and we conclude that the vertices $[a]$ and $[b]$ have 
		$$
		\dfrac{(q^{2e-2}-1)q^{2e-2} (q-1)^2}{q^2-q}= q^{2e-3}(q-1)(q^{2e-2}-1)
		$$
		common neighbours in ${Y(2e,K)}$, which completes the proof.
	\end{proof}

    Now Theorem \ref{thm:Main1} follows from Proposition \ref{prop1}, Proposition \ref{prop2}, Corollary \ref{cor:Partition}, Corollary \ref{cor:XTwoVerticesFromTheSameClass} and Corollary \ref{cor:XTwoVertices FromDifferentClasses}, and Theorem \ref{thm:Main2} follows from Proposition \ref{prop1}, Proposition \ref{prop3}, Corollary \ref{cor:Partition}, Proposition \ref{prop:YTwoVerticesFromTheSameClass} and Proposition \ref{prop:YTwoVerticesFromDifferentClasses}.

\section{Concluding remarks}\label{sec:Remarks}
In Corollary \ref{cor:Partition}, we showed the existence of a partition into classes $C(H,u)$, which is sufficient to prove Theorem \ref{thm:Main1} and Theorem \ref{thm:Main2}. However, the disadvantage of our results is the the canonical partition is given implicitly. The following problem naturally arises.

\begin{problem}
Let $m$ be the number of classes of the canonical partition of the graphs $X(2e,K)$ and $Y(2e,K)$ from Theorem \ref{thm:Main1} and Theorem \ref{thm:Main2}, which is equal to the number of $2e-1$-dimensional subspaces in a $2e$-dimensional vector space over a finite field $\mathbb{F}_q$.
How can we explicitly choose hyperplanes $H_1,H_2,\ldots,H_m$ and tuples $u_1,u_2,\ldots,u_m$ such that $C(H_1,u_1), C(H_2,u_2), \ldots, C(H_m,u_m)$ is the canonical partition of the graphs $X(2e,K)$ and $Y(2e,K)$?     
\end{problem}
    
\section*{Acknowledgements} \label{Ack}
This work is supported by Natural Science Fundation of Hebei Province (A2023205045).

\end{document}